\documentclass{article}

\usepackage[utf8]{inputenc}
\usepackage[T1]{fontenc}
\usepackage{amsmath}
\usepackage{amssymb}
\usepackage{amsthm}
\usepackage[english]{babel}

\usepackage{authblk}

\allowdisplaybreaks[3]

\usepackage[left=34mm,right=34mm,a4paper]{geometry}
\newtheorem{theorem}{Theorem}[section]

\newtheorem{lemma}[theorem]{Lemma}
\newtheorem{proposition}[theorem]{Proposition}
\newtheorem{example}[theorem]{Example}
\theoremstyle{definition}
\newtheorem{definition}[theorem]{Definition}
\theoremstyle{remark}
\newtheorem{remark}[theorem]{Remark}
\numberwithin{equation}{section}

\newcommand{\bracket}[1]{\left[#1\right]}
\newcommand{\para}[1]{\left(#1\right)}
\newcommand{\K}{\mathbb{K}}

\title{On $n$-ary Generalization of BiHom-Lie Algebras and BiHom-Associative Algebras}

\begin{document}

\author[1]{Abdennour Kitouni \thanks{abdennour.kitouni@gmail.com}}
\author[2]{Abdenacer Makhlouf \thanks{abdenacer.makhlouf@uha.fr}}
\author[3]{Sergei Silvestrov \thanks{sergei.silvestrov@mdh.se}}

\affil[1]{Université des Frères Mentouri - Constantine BP, 325 Route de Ain El Bey, 25017, Constantine, Algeria.}
\affil[2]{Université de Haute-Alsace, 4 rue des Frères Lumière, 68093 Mulhouse, France}
\affil[3]{Division of Applied Mathematics,
School of Education, Culture and Communication,
 M\"{a}lardalen University, Box 883, 72123 V{\"a}ster{\aa}s, Sweden.}

\date{}
\maketitle

\begin{abstract}
The aim of this paper is to introduce $n$-ary BiHom-algebras, generalizing BiHom-algebras. We  introduce an alternative concept of BiHom-Lie algebra called BiHom-Lie-Leibniz algebra and study  various type of  $n$-ary BiHom-Lie algebras and BiHom-associative algebras. We show that $n$-ary BiHom-Lie-Leibniz algebra can be represented by BiHom-Lie-Leibniz algebra through fundamental objets. Moreover, we provide some key constructions and study $n$-ary BiHom-Lie algebras induced by $(n-1)$-ary BiHom-Lie algebras.
\end{abstract}
\begin{small}
\end{small}

\section*{Introduction}
Hom-algebras first appeared in \cite{HLS} while investigating deformations of Witt and Virasoro algebras. authors discovered algebras similar to Lie algebras in which Jacobi identity is twisted by a linear map, such algebras were called Hom-Lie algebras. In \cite{ms:homstructure} this idea is  applied to generalize other algebras and the resulting algebras are studied. In \cite{LS:quasi-hom-lie}, authors introduced  further generalizations to include color Lie algebras and Lie superalgebras.

Ternary algebras and more generally $n$-ary Lie algebras first appeared in Nambu's generalization of Hamiltonian mechanics, using a ternary generalization of Poisson algebras. The mathematical formulation of Nambu mechanics is due to Takhtajan. Filippov, in \cite{Filippov:nLie} introduced $n$-Lie algebras then Kasymov \cite{Kasymov:nLie} deeper investigated their properties. This approach uses the interpretation of Jacobi identity expressing the fact that  the adjoint map is a derivation. There is also another type of $n$-ary Lie algebras, in which the $n$-ary Jacobi identity is the sum over $S_{2n-1}$ instead of $S_3$ in the binary case. One reason for studying such algebras was that $n$-Lie algebras introduced by Filippov were mostly rigid, and these algebras offered more possibilities to this regard.

Hom-type generalization of $n$-ary algebras were introduced in \cite{AtMaSi:GenNambuAlg}, by twisting the identities defining them using a set of $n-1$ linear maps, together with the particular case where all these maps are equal and are algebra morphisms. A way to generate examples of such algebras from non Hom-algebras of the same type is introduced. In \cite{YauHomNambuLie} and \cite{YauGenCom}, the author investigates some properties of $n$-ary Hom-algebras.

In \cite{MHomHopf}, authors looked at Hom-algebras from a category theoretical point of view, constructing a category on which algebras would be Hom-algebras. A generalization of this approach led to the discovery of BiHom-algebras in \cite{Bihom1}, called BiHom-algebras because the defining identities are twisted by two morphisms instead of only one for Hom-algebras.

The aim of this work is to introduce $n$-ary generalizations of BiHom-algebras and study their basic properties. Namely, we introduce two types of $n$-BiHom-Lie algebras, each one of them conserving a part of the properties of $n$-Lie algebras. We also define  totally BiHom-associative and partially BiHom-associative algebras.

The outline of this paper will be as follows: In the first section we present BiHom-Lie algebras as in \cite{Bihom1} and add an alternative definition which is not equivalent but reduces to the same definition in the case of Hom-Lie algebras. We also define two types of $n$-BiHom-Lie algebras and present their properties.
In the second section, we introduce $n$-ary totally BiHom-associative and partially BiHom-associative algebras and generalize to these cases the Yau twisting.
In the third section, we extend the construction of $(n+1)$-Lie algebras induced by $n$-Lie algebras (see \cite{Abramov} \cite{ams:ternary}, \cite{ams:n}, \cite{akms:ternary},\cite{km:n-ary},\cite{kms:n-hom})  to the case of $n$-ary BiHom-Lie algebras. Then we look at the conditions under which this construction is possible, in terms of image and kernel of the generalized trace map and the twisting maps. 

\section{BiHom-Lie algebras and $n$-BiHom-Lie algebras}
BiHom-algebras are a generalization of Hom-algebras regarding their construction using categories, which was introduced by Caenepeel and Goyvaerts in \cite{MHomHopf}. All the considered vector spaces are over a field of characteristic $0$. Everywhere hereafter, the notation $\widehat{x_i}$ in the arguments of an $n$-linear map means that $x_i$ is excluded, for example, we write $f\para{x_1,...,\widehat{x_i},...,x_n}$ for $f\para{x_1,...,x_{i-1},x_{i+1},...,x_n}$.
\begin{definition}[\cite{Bihom1}] 
A BiHom-Lie algebra is a vector space $A$ together with a bilinear map $\bracket{ \cdot,\cdot} : A^2 \to A$ and two linear maps $\alpha,\beta : A \to A$ satisfying the following conditions:
\begin{enumerate}
\item $\alpha \circ \beta = \beta \circ \alpha$.
\item $\forall x,y \in A, \alpha\para{\bracket{x,y}} = \bracket{\alpha(x),\alpha(y)}$ and $\beta\para{\bracket{x,y}} = \bracket{\beta(x),\beta(y)}$.
\item BiHom-skewsymmetry:  \[\forall x,y \in A, \bracket{\beta(x),\alpha(y)} = - \bracket{\beta(y),\alpha(x)}.\]
\item BiHom-Jacobi identity: \[\forall x,y,z \in A, \underset{x,y,z}{\circlearrowleft}\bracket{\beta^2(x),\bracket{\beta(y),\alpha(z)}} = 0.\]
\end{enumerate}
\end{definition}
One can also define a BiHom-Lie algebra using a generalization of the Jacobi identity under the form stating that the adjoint maps are derivations, we consider the following definition:
\begin{definition} 
A BiHom-Lie-Leibniz  algebra is a vector space $A$ together with a bilinear map $\bracket{ \cdot,\cdot} : A^2 \to A$ and two linear maps $\alpha,\beta : A \to A$ satisfying the following conditions:
\begin{enumerate}
\item $\alpha \circ \beta = \beta \circ \alpha$.
\item $\forall x,y \in A, \alpha\para{\bracket{x,y}} = \bracket{\alpha(x),\alpha(y)}$ and $\beta\para{\bracket{x,y}} = \bracket{\beta(x),\beta(y)}$.
\item BiHom-skewsymmetry  \[\forall x,y \in A, \bracket{\beta(x),\alpha(y)} = - \bracket{\beta(y),\alpha(x)}.\]
\item BiHom-Leibniz identity: \[\forall x,y,z \in A, \bracket{\beta^2(x),\bracket{\beta(y),\alpha(z)}} = \bracket{\bracket{\beta(x),\alpha(y)},\beta^2(z)} + \bracket{\beta^2(y),\bracket{\beta(x),\alpha(z)}}.\]
\end{enumerate}
\end{definition}

\begin{example}
We present here  few examples of BiHom-Lie-Leibniz  algebras. These examples were generated using a computer algebra software. Let $A$ be a $3$-dimensional vector space with basis $(e_1,e_2,e_3)$, the following brackets and matrices, defining linear maps in the considered basis, define on $A$ a BiHom-Lie-Leibniz  algebra structure:
\begin{enumerate}
\item
 $[\alpha]=\begin{pmatrix}0 & 0 & 0 \\ 0 & 1 & 0\\ 0 & 0 & a\end{pmatrix}$ ;
 $[\beta]= \begin{pmatrix}b & 0 & 0 \\ 0 & 0 & 0\\ 0 & 0 & 1\end{pmatrix}$ ;
 $[e_2,e_2]=c_1 e_2$ ; $[e_3,e_1]=c_2 e_1$ ; $[e_i,e_j]=0$  if  $(i,j)\neq (2,2),(3,1)$, where $a,b,c_1,c_2 \in\K$.
\item $[\alpha]=\begin{pmatrix}0 & 0 & 0 \\ 0 & a & 0\\ 0 & 0 & a^2\end{pmatrix}$ ;
 $[\beta]= \begin{pmatrix}b & 0 & 0 \\ 0 & 0 & 0\\ 0 & 0 & 0\end{pmatrix}$ ;
 $[e_2,e_2]=c e_3$ ;  $[e_i,e_j]=0$  if  $(i,j)\neq (2,2)$. Where $a,b,c \in\K$.
\item $[\alpha]=\begin{pmatrix}0 & 0 & 0 \\ 0 & a & 0\\ 0 & 0 & 1\end{pmatrix}$ ;
 $[\beta]= \begin{pmatrix}0 & 0 & 0 \\ 0 & b & 0\\ 0 & 0 & 1\end{pmatrix}$ ;
 $[e_1,e_1]=c_1 e_1$ ; $[e_1,e_2]=c_2 e_1$ ; $[e_1,e_3]=c_3 e_1$; $[e_2,e_1]=c_4 e_1$; $[e_3,e_1]=c_5 e_1$ ; $[e_i,e_j]=0$  for the remaining pairs $(i,j)$, where $a,b,c_k   \in\K ; (1 \leq k \leq 5)$.
\end{enumerate}
\end{example}

\begin{remark} 
The preceding definitions are not equivalent, and give two different classes of algebras. Unlike Lie algebras, such algebras are not skewsymmetric, and thus BiHom-Jacobi and BiHom-Leibniz identities are not equivalent. However, both of them reduce to Lie algebras when the considered morphisms $\alpha$ and $\beta$ are equal to the identity map. Also, if $\alpha$ and $\beta$ are equal and surjective, algebras given by both definitions are multiplicative Hom-Lie algebras.
\end{remark}

\begin{remark} 
 BiHom-Jacobi and BiHom-Leibniz identities are  equivalent if one assumes that the bracket is in addition skewsymmetric in the usual sense. \\ Notice that BiHom-skewsymmetry is equivalent in this case to $[\beta(x),\alpha(y)]=[\alpha(x),\beta(y)]$.
\end{remark}

Generalizing $n$-Lie algebras following the same process gives a class of algebra which does not reduce to BiHom-Lie algebras for $n=2$. We use instead a different form for the fundamental identity to construct our generalization, namely:

\begin{definition} 
An $n$-BiHom-Lie algebra is a vector space $A$, equipped with an $n$-linear operation $\bracket{\cdot,...,\cdot}$ and two linear maps $\alpha$ and $\beta$ satisfying the following conditions:
\begin{enumerate}
\item $\alpha \circ \beta = \beta \circ \alpha$.
\item $\forall x_1,...,x_n \in A, \alpha\para{\bracket{x_1,...,x_n}} = \bracket{\alpha(x_1),...,\alpha(x_n)}$ and $\beta\para{\bracket{x_1,...,x_n}} = \bracket{\beta(x_1),...,\beta(x_n)}$.
\item BiHom-skewsymmetry: $\forall x_1,...,x_n \in A,\forall \sigma \in S_n,$
\[
\bracket{\beta(x_1),...,\beta(x_{n-1}),\alpha(x_n)} = Sgn(\sigma) \bracket{\beta(x_{\sigma(1)}),...,\beta(x_{\sigma(n-1)}),\alpha(x_{\sigma(n)})}.
\]
\item $n$-BiHom-Jacobi identity: $\forall x_1,..,x_{n-1},y_1,...,y_n \in A$,
\begin{align*}
&\bracket{\beta^2(x_1),...,\beta^2(x_{n-1}),\bracket{\beta(y_1),...,\beta(y_{n-1}),\alpha(y_n)}}\\
& =\sum_{k=1}^n (-1)^{n-k} \bracket{\beta^2(y_1),...,\widehat{\beta^2(y_k)},...,\beta^2(y_n),\bracket{\beta(x_1),...,\beta(x_{n-1}),\alpha(y_k)}}.
\end{align*}
\end{enumerate}
\end{definition}

The straightforward generalization of $n$-Lie algebras leads to a class of $n$-ary algebras which reduces to BiHom-Lie-Leibniz  algebras for $n=2$, we will call them $n$-BiHom-Lie-Leibniz  algebras, and define them by:

\begin{definition}
A $n$-BiHom-Lie-Leibniz  algebra is a vector space $A$, equipped with an $n$-linear operation $\bracket{\cdot,...,\cdot}$ and two linear maps $\alpha$ and $\beta$ satisfying the following conditions:
\begin{enumerate}
\item $\alpha \circ \beta = \beta \circ \alpha$.
\item $\forall x_1,...,x_n \in A, \alpha\para{\bracket{x_1,...,x_n}} = \bracket{\alpha(x_1),...,\alpha(x_n)}$ and $\beta\para{\bracket{x_1,...,x_n}} = \bracket{\beta(x_1),...,\beta(x_n)}$.
\item BiHom-skewsymmetry : $\forall x_1,...,x_n \in A,\forall \sigma \in S_n,$
 \[ \bracket{\beta(x_1),...,\beta(x_{n-1}),\alpha(x_n)}
= Sgn(\sigma) \bracket{\beta(x_{\sigma(1)}),...,\beta(x_{\sigma(n-1)}),\alpha(x_{\sigma(n)})}.
\]
\item BiHom-Nambu identity: $\forall x_1,..,x_{n-1},y_1,...,y_n \in A,$
\begin{align*}
 &\bracket{\beta^2(x_1),...,\beta^2(x_{n-1}),\bracket{\beta(y_1),...,\beta(y_{n-1}),\alpha(y_n)}} \\ &=\sum_{k=1}^n \bracket{\beta^2(y_1),...,\bracket{\beta(x_1),...,\beta(x_{n-1}),\alpha(y_k)},...,\beta^2(y_n)}.
 \end{align*}
\end{enumerate}
\end{definition}

Now, we introduce the notions of morphisms, subalgebras and ideals of such algebras. After that we will extend some properties of $n$-Lie algebras and $n$-Hom-Lie algebras to this case, some of these properties hold only for one of the classes of algebras defined above.

Let $(A,\bracket{\cdot,...,\cdot},\alpha,\beta)$ and $(A,\bracket{\cdot,...,\cdot},\alpha',\beta')$ be $n$-BiHom-Lie algebras (resp. Leibniz)

\begin{definition} 
A linear map $f : A \to B$ is said to be an $n$-BiHom-Lie algebra morphism if it satisfies the following properties:
\begin{itemize}
\item $ \alpha'\circ f = f \circ \alpha$ and $\beta' \circ f = f \circ \beta$,
\item $\forall x_1,...,x_n \in A, f\para{\bracket{x_1,...,x_n}} = \bracket{f(x_1),...,f(x_n)}$.
\end{itemize}
\end{definition}

\begin{definition} 
A subset $S \subseteq A$ is a subalgebra if $\alpha(S)\subseteq S$, $\beta(S)\subseteq S$ and $\forall x_1,...,x_n \in S, \bracket{x_1,...,x_n} \in S$. It is said to be an ideal if $\alpha(S)\subseteq S$, $\beta(S)\subseteq S$ and $\forall x_1,...,x_{n-1} \in A, s \in S, \bracket{x_1,...,x_{n-1},s} \in S$.
\end{definition}

\subsection{Fundamental objects and Basic algebra}
In $n$-Lie algebras, the adjoint maps, and more generally actions and representations of an $n$-Lie algebra are defined by giving $n-1$ elements of the algebra. This leads to the notion of fundamental objects and basic Lie algebra \cite{Dal_Takh}. We generalize these constructions to a $n$-BiHom-Lie-Leibniz  algebra, under some conditions on the linear maps $\alpha$ and $\beta$ of this algebra.

Let $\para{A,\bracket{\cdot,...,\cdot},\alpha,\beta}$ be a $n$-BiHom-Lie-Leibniz  algebra such that $\alpha$ is bijective and $\beta$ is surjective. Notice that under these conditions, the algebra's bracket becomes skewsymmetric in its $(n-1)$ first arguments.
\begin{definition}
Fundamental object of the BiHom-algebra $\para{A,\bracket{\cdot,...,\cdot},\alpha,\beta}$ are elements of the $(n-1)$-th exterior power of $A$, that is $\wedge^{n-1}A$.

We also define, for all $X = x_1\wedge ... \wedge x_{n-1}$, $Y = y_1\wedge ... \wedge y_{n-1}$ in $\wedge^{n-1}A$ and $z\in A$ the following operations:
\begin{itemize}
\item The action of fundamental objects on $A$:
\[ X \cdot z = ad_X(z) = \bracket{x_1,...,x_{n-1},z}. \]
\item The multiplication of fundamental objects:
\[ X \cdot Y =\bracket{X,Y} = \sum_{i=1}^{n-1} \beta^2\circ \alpha^{-1}(y_1)\wedge ... \wedge \bracket{x_1,...,x_{n-1},y_i }\wedge...\wedge \beta^2\circ \alpha^{-1}(y_{n-1}). \]
\item The linear maps $\bar{\alpha},\bar{\beta} : \wedge^{n-1}A \to \wedge^{n-1}A$:
\[\bar{\alpha}(X) = \alpha(x_1)\wedge...\wedge \alpha(x_{n-1}) \text{ and } \bar{\beta}(X) = \beta(x_1)\wedge...\wedge \beta(x_{n-1}).\]
\end{itemize}
We extend the preceding definitions to the whole set of fundamental objects by linearity.
\end{definition}

The definition above of the multiplication of fundamental objects may seem unnatural, the motivation behind it is that using this definition, one can write the BiHom-Nambu identity in the following form:
\[ \bar{\beta}^2(X)\cdot\para{\bar{\beta}(Y) \cdot \alpha(z)} = \para{\bar{\beta}(X)\cdot \bar{\alpha}(Y)} \cdot \beta^2(z) + \bar{\beta}^2(Y) \cdot \para{\bar{\beta}(X) \cdot \alpha(z)}, \forall X, Y \in \wedge^{n-1}A, \forall z \in A. \]
The multiplication of fundamental objects satisfies the following property, generalizing a similar property in $n$-Lie algebras case:
\begin{theorem}
Let $\para{A,\bracket{\cdot,...,\cdot},\alpha,\beta}$ be a $n$-BiHom-Lie-Leibniz  algebra such that $\alpha$ is bijective and $\beta$ is surjective. The set of fundamental objects, which we denote by $L(A)$, equipped with the multiplication of fundamental objects and the maps $\bar{\alpha}$ and $\bar{\beta}$ defined above, is a BiHom-Lie-Leibniz  algebra.
\end{theorem}

\begin{proof}
Let $X=x_1\wedge ... \wedge x_{n-1}$ and $Y=y_1\wedge...\wedge y_{n-1}$, then we have:
\begin{align*}
\bar{\alpha}\circ \bar{\beta}(X) &= \bar{\alpha}\para{\beta(x_1)\wedge ... \wedge \beta(x_{n-1})}\\
&= \alpha\circ\beta(x_1)\wedge...\wedge\alpha\circ\beta(x_{n-1})\\
&= \beta\circ\alpha(x_1)\wedge...\wedge \beta\circ\alpha(x_{n-1})\\
&= \bar{\beta}\para{\alpha(x_1)\wedge...\wedge\alpha(x_{n-1})}\\
&=\bar{\beta}\circ \bar{\alpha}(X).
\end{align*}

\begin{align*}
\bar{\alpha}\para{\bracket{X,Y}} &=\sum_{i=1}^{n-1}\bar{\alpha}\para{\beta^2\circ\alpha^{-1}(y_1)\wedge ... \wedge X \cdot y_i \wedge ... \wedge\beta^2\circ\alpha^{-1}(y_{n-1})}\\
&= \sum_{i=1}^{n-1}\para{\alpha\circ\beta^2\circ\alpha^{-1}(y_1)\wedge ... \wedge \alpha(X \cdot y_i) \wedge ... \wedge \alpha\circ\beta^2\circ\alpha^{-1}(y_{n-1})}\\
&= \sum_{i=1}^{n-1}\para{\beta^2\circ\alpha^{-1}(\alpha(y_1))\wedge ... \wedge \bracket{\alpha(x_1),...,\alpha(x_{n-1}),\alpha(y_i)} \wedge ... \wedge \beta^2\circ\alpha^{-1}(\alpha(y_{n-1}))}\\
&= \bracket{\bar{\alpha}(X),\bar{\alpha}(Y)},
\end{align*}
that is $\bar{\alpha}$ is an algebra morphism. One can show, in the same way, that $\bar{\beta}$ is a morphism.
\begin{align*}
\bracket{\bar{\beta}(X),\bar{\alpha}(X)}&= \sum_{i=1}^{n-1} \beta^2\circ \alpha^{-1}(\alpha(x_1))\wedge ... \wedge \bracket{\beta(x_1),...,\beta(x_{n-1}),\alpha(x_i)}\wedge ... \wedge \beta^2\circ \alpha^{-1}(\alpha(x_{n-1}))\\
&=\sum_{i=1}^{n-1} \beta^2\circ \alpha^{-1}(\alpha(x_1))\wedge ... \wedge 0 \wedge ... \wedge \beta^2\circ \alpha^{-1}(\alpha(x_{n-1}))\\
&=0.
\end{align*}
For all $X=x_1\wedge ... \wedge x_{n-1}$, $Y=y_1\wedge ... \wedge y_{n-1}$ and $Z=z_1\wedge ... \wedge z_{n-1}$, we have
\begin{align*}
&\bracket{\bar{\beta}^2(X),\bracket{\bar{\beta}(Y),\bar{\alpha}(Z)}} = \sum_{i=1}^{n-1} \bracket{\bar{\beta}^2(X), \para{\beta^2(z_1)\wedge...\wedge \bracket{\beta(y_1),...,\beta(y_{n-1}),\alpha(z_i)}\wedge...\wedge \beta^2(y_{n-1}) }  }\\
&= \sum_{i=1}^{n-1} \sum_{j=1}^{i-1} \big(\beta^2\circ\alpha^{-1}\circ\beta^2(z_1)\wedge...\wedge \beta^2(X\cdot z_j)\wedge \\
&\qquad \wedge...\wedge \beta^2\circ\alpha^{-1}\para{\bracket{\beta(y_1),...,\beta(y_{n-1}),\alpha(z_i)} }\wedge...\wedge \beta^2\circ\alpha^{-1}\circ\beta^2(y_{n-1}) \big) \\
&+ \sum_{i=1}^{n-1} \sum_{j=i+1}^{n-1} \big( \beta^2\circ\alpha^{-1}\circ\beta^2(z_1)\wedge...\wedge \beta^2\circ\alpha^{-1}\para{\bracket{\beta(y_1),...,\beta(y_{n-1}),\alpha(z_i)} }\\
&\qquad \wedge...\wedge \beta^2(X\cdot z_j)\wedge...\wedge \beta^2\circ\alpha^{-1}\circ\beta^2(y_{n-1}) \big)\\
&+ \sum_{i=1}^{n-1} \big( \beta^2\circ\alpha^{-1}\circ\beta^2(z_1)\wedge...\wedge \bar{\beta}^2\para{X} \cdot \bracket{\beta(y_1),...,\beta(y_{n-1}),\alpha(z_i)}\wedge...\wedge \beta^2\circ\alpha^{-1}\circ\beta^2(y_{n-1}) \big) \\
&=  \sum_{i=1}^{n-1} \sum_{j=1}^{i-1} \big( \beta^2\circ\alpha^{-1}\circ\beta^2(z_1)\wedge...\wedge \beta^2(X\cdot z_j)\wedge...\wedge \beta^2\circ\alpha^{-1}\para{\bracket{\beta(y_1),...,\beta(y_{n-1}),\alpha(z_i)} } \wedge \\
&\qquad \wedge...\wedge \beta^2\circ\alpha^{-1}\circ\beta^2(y_{n-1}) \big) \\
& -  \sum_{i=1}^{n-1} \sum_{j=1}^{i-1} \big( \beta^2\circ\alpha^{-1}\circ\beta^2(z_1)\wedge...\wedge \beta^2(X\cdot z_j)\wedge...\wedge \beta^2\circ\alpha^{-1}\para{\bracket{\beta(y_1),...,\beta(y_{n-1}),\alpha(z_i)} }\wedge \\
&\qquad \wedge...\wedge \beta^2\circ\alpha^{-1}\circ\beta^2(y_{n-1}) \big) \\
&+\sum_{i=1}^{n-1} \para{\beta^2\circ\alpha^{-1}\circ\beta^2(z_1)\wedge...\wedge \bar{\beta}^2\para{X} \cdot \para{\bar{\beta}(Y) \cdot \alpha(z)}\wedge...\wedge \beta^2\circ\alpha^{-1}\circ\beta^2(y_{n-1}) } \\
&= \sum_{i=1}^{n-1} \para{\beta^2\circ\alpha^{-1}\circ\beta^2(z_1)\wedge...\wedge \bar{\beta}^2\para{X} \cdot \para{\bar{\beta}(Y) \cdot \alpha(z)}\wedge...\wedge \beta^2\circ\alpha^{-1}\circ\beta^2(y_{n-1}) }\\
&= \sum_{i=1}^{n-1} \para{\beta^2\circ\alpha^{-1}\circ\beta^2(z_1)\wedge...\wedge  \para{\bar{\beta}(X) \cdot \bar{\alpha}(Y)} \cdot \beta^2\para{z} \wedge...\wedge \beta^2\circ\alpha^{-1}\circ\beta^2(y_{n-1}) }\\
&+ \sum_{i=1}^{n-1} \para{\beta^2\circ\alpha^{-1}\circ\beta^2(z_1)\wedge...\wedge \bar{\beta}^2\para{Y} \cdot \para{\bar{\beta}(X) \cdot \alpha(z)}\wedge...\beta^2\circ\alpha^{-1}\circ\wedge \beta^2(y_{n-1}) }\\
&= \bracket{\bracket{\bar{\beta}(X),\bar{\alpha}(Y)},\bar{\beta}^2(Z)} + \bracket{\bar{\beta}^2(Y),\bracket{\bar{\beta}(X),\bar{\alpha}(Z)}}.
\end{align*}
\end{proof}

\subsection{Algebra Twisting}
The following result gives a method to construct BiHom-Lie algebras, and more generally $n$-BiHom-Lie algebras, starting from a Lie or $n$-Lie algebra and two commuting algebra endomorphisms. It generalizes results from \cite{YauHomHom} and \cite{AtMaSi:GenNambuAlg} for Hom-algebras.

\begin{theorem} \label{YauTwist}
Let $(A,\bracket{\cdot,...,\cdot})$ be an $n$-Lie algebra, and let $\alpha,\beta : A \to A$ be algebra morphisms such that $\alpha \circ \beta = \beta \circ \alpha$. The algebra $(A, \bracket{\cdot,...,\cdot}_{\alpha\beta},\alpha,\beta)$, where $\bracket{\cdot,...,\cdot}_{\alpha\beta}$ is defined by
\[\bracket{x_1,...,x_n}_{\alpha\beta} = \bracket{\alpha(x_1),...,\alpha(x_{n-1}),\beta(x_n)},\]
is an $n$-BiHom-Lie algebra.
\end{theorem}

\begin{proof}
The maps $\alpha$ and $\beta$ commute, by hypothesis, we show that they are algebra morphisms, for all $x_1,...,x_n \in A$, we have:
\begin{align*}
\alpha\para{\bracket{x_1,...,x_n}_{\alpha\beta}} &= \alpha\para{\bracket{\alpha(x_1),...,\alpha(x_{n-1},\beta(x_n)}}\\
&= \bracket{\alpha^2(x_1),...,\alpha^2(x_{n-1},\alpha\circ\beta(x_n)}\\
&=\bracket{\alpha^2(x_1),...,\alpha^2(x_{n-1},\beta\circ\alpha(x_n)}\\
&=\bracket{\alpha(x_1),...,\alpha(x_n)}_{\alpha\beta}.
\end{align*}
One can prove, in a very similar way, that $\beta$ is also a morphism.
We have also:
\begin{align*}
\bracket{\beta(x_{\sigma(1)}),...,\beta(x_{\sigma(n-1)}),\alpha(x_{\sigma(n)})}_{\alpha\beta} &= \bracket{\alpha\circ\beta(x_{\sigma(1)}),...,\alpha\circ\beta(x_{\sigma(n-1)}),\beta\circ \alpha(x_{\sigma(n)})}\\
&= \alpha\circ\beta\para{\bracket{x_{\sigma(1)},...,x_{\sigma(n)}}}\\
&= Sgn(\sigma) \alpha\circ\beta\para{\bracket{x_1,...,x_n}} \\
&= Sgn(\sigma) \bracket{\alpha\circ\beta(x_{1}),...,\alpha\circ\beta(x_{n-1}),\beta\circ \alpha(x_{n})}\\
&=Sgn(\sigma) \bracket{\beta(x_1),...,\beta(x_{n-1}),\alpha(x_n)}_{\alpha\beta}.
\end{align*}
For all $x_1,...,x_{n-1},y_1,...,y_n \in A$, we have:
\begin{align*}
&\bracket{\beta^2(x_1),...,\beta^2(x_{n-1}),\bracket{\beta(y_1),...,\beta(y_{n-1}),\alpha(y_n)}_{\alpha\beta}}_{\alpha\beta}\\ =&\bracket{\alpha\circ\beta^2(x_1),...,\alpha\circ\beta^2(x_{n-1}),\beta\para{\bracket{\alpha\circ\beta(y_1),...,\alpha\circ\beta(y_{n-1}),\beta\circ\alpha(y_n)}}}\\
= & \alpha\circ\beta^2\para{\bracket{x_1,...,x_{n-1},\bracket{y_1,...,y_n}}} \\
= & \sum_{k=1}^n \alpha\circ\beta^2\para{\bracket{y_1,...,\bracket{x_1,...,x_{n-1},y_k},...,y_n}}\\
= & \sum_{k=1}^n (-1)^{n-k} \alpha\circ\beta^2\para{\bracket{y_1,...,\widehat{y_k},...,y_n,\bracket{x_1,...,x_{n-1},y_k}}}\\
= &  \sum_{k=1}^n (-1)^{n-k} \bracket{\alpha\circ\beta^2(y_1),...,\widehat{\alpha\circ\beta^2(y_k)},...,\alpha\circ\beta^2(y_n),\alpha\circ\beta^2\para{\bracket{x_1,...,x_{n-1},y_k}}}\\
= & \sum_{k=1}^n (-1)^{n-k} \bracket{\alpha\circ\beta^2(y_1),...,\widehat{\alpha\circ\beta^2(y_k)},...,\alpha\circ\beta^2(y_n),\beta\para{\bracket{\alpha\circ\beta(x_1),...,\alpha\circ\beta(x_{n-1}),\beta\circ\alpha(y_k)}}}\\
=& \sum_{k=1}^n (-1)^{n-k} \bracket{\beta^2(y_1),...,\widehat{\beta^2(y_k)},...,\beta^2(y_n),\bracket{\beta(x_1),...,\beta(x_{n-1}),\alpha(y_k)}_{\alpha\beta}}_{\alpha\beta}.
\end{align*}
\end{proof}

\begin{example}
We consider the $4$-dimensional $3$-Lie algebra defined with respect to a basis $(e_1,e_2,e_3,e_4)$, by:
\[ [e_1,e_2,e_3]=e_4 ; [e_1,e_2,e_4]=e_3 ; [e_1,e_3,e_4]=e_2 ; [e_2,e_3,e_4]=e_1. \]

We have two morphisms $\alpha$, $\beta$ of this algebra defined, with respect to the same basis, by:
\[ [\alpha]=\begin{pmatrix}-1 & 0 & 0 & 0 \\ 0 & 1 & 0 & 0 \\ 0 & 0 & 1 & 0 \\ 0 & 0 & 0 & -1 \end{pmatrix} ;
[\beta]= \begin{pmatrix} 0 & 0 & 0 & -1 \\ 0 & 0 & -1 & 0 \\ 0 & -1 & 0 & 0 \\ 1 & 0 & 0 & 0 \end{pmatrix}. \]
One can easily check that $\alpha$ and $\beta$ commute, then one may  to construct a $n$-BiHom-Lie algebra using Theorem \ref{YauTwist}, we get the following bracket:
\begin{align*}
[e_1,e_2,e_1]&= - e_3 ; [e_1,e_2,e_2]=e_4 ; [e_1,e_3,e_1]=-e_2 ; [e_1,e_3,e_3]=-e_4 ; \\ [e_1,e_4,e_2]&=e_2 ; [e_2,e_1,e_1]=-e_3 ; [e_1,e_2,e_1]=e_3 ; [e_2,e_1,e_2]=e_4 ; \\ [e_2,e_3,e_1]&=e_1 ; [e_2,e_3,e_4]=e_4 ; [e_2,e_4,e_2]=-e_1 ; [e_2,e_4,e_4]=-e_3 ; \\ [e_3,e_1,e_1]&=-e_2 ; [e_3,e_1,e_3]=e_4 ; [e_3,e_2,e_1]=e_1 ; [e_3,e_2,e_4]=e_4 ; \\ [e_3,e_4,e_3]&=-e_1 ; [e_3,e_4,e_4]=-e_2 ; [e_4,e_1,e_2]=-e_2 ; [e_4,e_1,e_3]=-e_3 ; \\ [e_4,e_2,e_2]&=e_1 ; [e_4,e_2,e_4]=-e_3 ; [e_4,e_3,e_3]=-e_1 ; [e_4,e_3,e_4]=-e_2 ;
\end{align*}
\end{example}



\section{Associative type $n$-ary BiHom-algebras}
In this section, we present  generalizations of $n$-ary algebras of associative type, namely  totally BiHom-associative and partially BiHom-associative algebras, we also give a generalization of the Yau twist corresponding to these structures. These algebras also generalize the $n$-ary Hom-algebra of associative type introduced in \cite{AtMaSi:GenNambuAlg}.

\begin{definition}
An $n$-ary  totally BiHom-associative algebra is a vector space $A$ together with an $n$-linear map $m: A^n \to A$ and two linear maps $\alpha$, $\beta$ satisfying the following conditions:
\begin{itemize}
\item $\alpha \circ \beta = \beta \circ \alpha$.
\item $\alpha\para{m\para{x_1,...,x_n}} = m\para{\alpha\para{x_1},...,\alpha\para{x_n}}, \forall x_1,...,x_n \in A$.
\item $\beta\para{m\para{x_1,...,x_n}} = m\para{\beta\para{x_1},...,\beta\para{x_n}}, \forall x_1,...,x_n \in A$.
\item Total BiHom-associativity: $\forall x_1,...,x_{2n-1}\in A, \forall i,j: 1\leq i,j \leq n$
\begin{align*}
m&\para{\alpha(x_1),...,\alpha(x_{i-1}),m(x_i,....,x_{n+i-1}),\beta(x_{n+i}),...,\beta(x_{2n-1})} \\
&= m\para{\alpha(x_1),...,\alpha(x_{j-1}),m(x_j,....,x_{n+j-1}),\beta(x_{n+j}),...,\beta(x_{2n-1})}.
\end{align*}
\end{itemize}
\end{definition}

\begin{definition}
An $n$-ary  partially BiHom-associative algebra is a vector space $A$ together with an $n$-linear map $m: A^n \to A$ and two linear maps $\alpha$, $\beta$ satisfying the following conditions:
\begin{itemize}
\item $\alpha \circ \beta = \beta \circ \alpha$.
\item $\alpha\para{m\para{x_1,...,x_n}} = m\para{\alpha\para{x_1},...,\alpha\para{x_n}}, \forall x_1,...,x_n \in A$.
\item $\beta\para{m\para{x_1,...,x_n}} = m\para{\beta\para{x_1},...,\beta\para{x_n}}, \forall x_1,...,x_n \in A$.
\item  Partial BiHom-associativity: $\forall x_1,...,x_{2n-1}\in A,$
\begin{equation*}
\sum_{i=1}^n m\para{\alpha(x_1),...,\alpha(x_{i-1}),m(x_i,....,x_{n+i-1}),\beta(x_{n+i}),...,\beta(x_{2n-1})} =0.
\end{equation*}
\end{itemize}
\end{definition}

\begin{remark}
In the definitions above, the particular case where $\alpha=\beta$ leads us to the definitions of $n$-ary totally Hom-associative (resp. partially Hom-associative) algebras. Choosing $\alpha=\beta=Id_A$ gives the definitions of $n$-ary totally associative (resp. partially associative) algebras.
\end{remark}

Now, we introduce as for $n$-Lie algebras, a generalization of the Yau twist, allowing us to construct $n$-ary BiHom-algebra of associative type given an $n$-ary algebra of associative type and two linear maps satisfying some conditions.

\begin{proposition}
Let $(A,m)$ be an $n$-ary totally associative (resp. partially associative) algebra, and let $\alpha$, $\beta$ be two algebra endomorphisms of $A$ satisfying $\alpha \circ \beta = \beta \circ \alpha$. We define $m_{\alpha,\beta}:A^n \to A$ by:
\[ m_{\alpha,\beta}\para{x_1,...,x_n}=m\para{\alpha^{n-1}(x_1),\alpha^{n-2}\circ \beta(x_2),...,\alpha\circ\beta^{n-2}(x_{n-1}),\beta^{n-1}(x_n)}. \]
Then $(A,m_{\alpha,\beta},\alpha^{n-1},\beta^{n-1})$, which we will denote by $A_{\alpha,\beta}$, is an $n$-ary totally  (resp. partially) BiHom-associative algebra.
\end{proposition}

\begin{proof}
The first three conditions come down from $\alpha$ and $\beta$ being commuting algebra morphisms, we only have to show the total (resp. partial) BiHom-associativity. Let $x_1,...,x_{2n-1} \in A$, we have, for $i : 1\leq i \leq n$:
\begin{align*}
m_{\alpha,\beta}&\para{\alpha^{n-1}(x_1),...,\alpha^{n-1}(x_{i-1}),m_{\alpha,\beta}(x_i,...,x_{i+n-1}),\beta^{n-1}(x_{n+i}),...,\beta^{n-1}(x_{2n-1}  }\\
&=m(   \alpha^{n-1}\circ\alpha^{n-1}(x_1),...,\alpha^{n-i+1}\circ\beta^{i-2}\circ\alpha^{n-1}(x_{i-1}),\alpha^{n-i}\circ\beta^{i-1}\circ\\
&\circ m(\alpha^{n-1}(x_i),...,\beta^{n-1}(x_{i+n-1}) ),\alpha^{n-i-1}\circ\beta^{i}\circ\beta^{n-1}(x_{n+i}),...,\beta^{n-1}\circ\beta^{n-1}(x_{2n-1})  )\\
&= m(   \alpha^{2n-2}(x_1),...,\alpha^{2n-i}\circ\beta^{i-2}(x_{i-1}),m(\alpha^{2n-i-1}\circ\beta^{i-1}(x_i),...,\times\\
&\times\alpha^{n-i}\circ\beta^{n+i-2}(x_{n+i-1}) ), \alpha^{n-i-1}\circ\beta^{n+i-1}(x_{n+i}),...,\beta^{2n-2}(x_{2n-1})  ).\\
\end{align*}
For  $j\neq i$, we get, following the same process:
\begin{align*}
m_{\alpha,\beta}&\para{\alpha^{n-1}(x_1),...,\alpha^{n-1}(x_i),...,\alpha^{n-1}(x_{i-1}),m_{\alpha,\beta}(x_i,...,x_{i+n-1}),\beta^{n-1}(x_{n+i}),...,\beta^{n-1}(x_{2n-1}  }\\
&= m(   \alpha^{2n-2}(x_1),...,\alpha^{2n-j-1}\circ\beta^{j-1}(x_j),...,\alpha^{2n-j}\circ\beta^{j-2}(x_{j-1}),m(\alpha^{2n-j-1}\circ\beta^{j-1}(x_j)\times\\
&\times ,...,\alpha^{n-j}\circ\beta^{n+j-2}(x_{n+j-1}) ), \alpha^{n-j-1}\circ\beta^{n+j-1}(x_{n+j}),...,\beta^{2n-2}(x_{2n-1})  ).\\
\end{align*}
Using total associativity of $m$, we get:
\begin{align*}
m&(   \alpha^{2n-2}(x_1),...,\alpha^{2n-j-1}\circ\beta^{j-1}(x_j),...,\alpha^{2n-j}\circ\beta^{j-2}(x_{j-1}), m(\alpha^{2n-j-1}\circ\beta^{j-1}(x_j)\times\\
&\times,...,\alpha^{n-j}\circ\beta^{n+j-2}(x_{n+j-1}) ), \alpha^{n-j-1}\circ\beta^{n+j-1}(x_{n+j}),...,\beta^{2n-2}(x_{2n-1})  )\\
&= m(   \alpha^{2n-2}(x_1),...,\alpha^{2n-i}\circ\beta^{i-2}(x_{i-1}),m(\alpha^{2n-i-1}\circ\beta^{i-1}(x_i),...,\alpha^{n-i}\circ\beta^{n+i-2}(x_{n+i-1}) ),\times\\
&\times \alpha^{n-i-1}\circ\beta^{n+i-1}(x_{n+i}),...,\beta^{2n-2}(x_{2n-1})  )\\
&=m_{\alpha,\beta}\para{\alpha^{n-1}(x_1),...,\alpha^{n-1}(x_{i-1}),m_{\alpha,\beta}(x_i,...,x_{i+n-1}),\beta^{n-1}(x_{n+i}),...,\beta^{n-1}(x_{2n-1}  }.\\
\end{align*}

Now we show the partial associativity condition for the relevant case. Let $x_1,...,x_{2n-1} \in A$, we have:
\begin{align*}
\sum_{i=1}^n m_{\alpha,\beta}&\para{\alpha^{n-1}(x_1),...,\alpha^{n-1}(x_{i-1}),m_{\alpha,\beta}(x_i,...,x_{i+n-1}),\beta^{n-1}(x_{n+i}),...,\beta^{n-1}(x_{2n-1}  }\\
&=\sum_{i=1}^n m(   \alpha^{n-1}\circ\alpha^{n-1}(x_1),...,\alpha^{n-i+1}\circ\beta^{i-2}\circ\alpha^{n-1}(x_{i-1}),\alpha^{n-i}\circ\beta^{i-1}\circ\\
&\circ m(\alpha^{n-1}(x_i),...,\beta^{n-1}(x_{i+n-1}) ),\alpha^{n-i-1}\circ\beta^{i}\circ\beta^{n-1}(x_{n+i}),...,\beta^{n-1}\circ\beta^{n-1}(x_{2n-1})  )\\
&= \sum_{i=1}^n m(   \alpha^{2n-2}(x_1),...,\alpha^{2n-i}\circ\beta^{i-2}(x_{i-1}),m(\alpha^{2n-i-1}\circ\beta^{i-1}(x_i),...,\times\\
&\times\alpha^{n-i}\circ\beta^{n+i-2}(x_{n+i-1}) ), \alpha^{n-i-1}\circ\beta^{n+i-1}(x_{n+i}),...,\beta^{2n-2}(x_{2n-1})  )\\
&=0.
\end{align*}

\end{proof}

The definitions and result above can be extended to some variants of $n$-ary algebras of associative type. Namely $n$-ary weak totally associative algebras, where the total associativity holds only for $i=1$, $j=n$ in the definition above, and $n$-ary alternate partially associative algebras, where the the terms in the sum defining partial associativity are multiplied by $(-1)^i$.




\section{$(n+1)$-BiHom-Lie algebras induced by $n$-BiHom-Lie algebras}
The aim of this section is to extend, to $n$-BiHom-Lie algebras, the construction of $(n+1)$-Hom-Lie algebras from $n$-Hom-Lie algebras introduced in \cite{ams:n}, and to see under which conditions such a generalization is possible. First, we give some definitions and lemmas allowing to reach our goal.

\begin{definition} [\cite{ams:ternary}, \cite{ams:n}]
Let $A$ be a vector space, $\phi : A^n \to A$ be an $n$-linear map and $\tau$ be a linear form. We define the $(n+1)$-linear map $\phi_\tau$ by:
\[\forall x_1,...,x_{n+1} \in A, \phi_\tau\para{x_1,...,x_{n+1}} = \sum_{i=1}^{n+1} (-1)^{i-1} \tau(x_i)\phi\para{x_1,...,\widehat{x_i},...,x_{n+1}}. \]
\end{definition}
 As in the case of $n$-Lie or $n$-Hom-Lie algebras, we focus on linear forms $\tau$ satisfying a generalization of the properties of the trace of matrices, namely, we consider the following definition:
\begin{definition} 
Let $A$ be a vector space, $\phi : A^n \to A$ be an $n$-linear map, let $\tau$ be a linear form and $\alpha,\beta : A \to A$ be linear maps. The map $\tau$ is said to be an $(\alpha,\beta)$-twisted $\phi$-trace if if satisfies the following condition:
\[ \forall x_1,...,x_n \in A, \tau\para{\phi\para{\beta(x_1),...,\beta(x_{n-1}),\alpha(x_n)}} = 0. \]
If the maps $\phi$, $\alpha$ and $\beta$ above are clear from the context, we will simply refer to such linear forms by twisted traces.
\end{definition}

\begin{lemma}
Let $A$ be a vector space, $\phi : A^n \to A$ be an $n$-linear map,  $\tau$ be a linear form and $\alpha,\beta : A \to A$ be linear maps. If $\tau$ is an $(\alpha,\beta)$-twisted $\phi$-trace satisfying the condition $\tau(\alpha(x))\beta(y) = \tau(\beta(x))\alpha(y)$, $\forall x,y \in A$, then $\tau$ is an $(\alpha,\beta)$-twisted $\phi_\tau$-trace.
\end{lemma}

\begin{proof}
Let $x_1,...,x_{n+1} \in A$, we have:
\begin{align*}
&\tau\para{\phi_\tau\para{\beta(x_1),...,\beta(x_n),\alpha(x_{n+1})}} \\
=& \sum_{i=1}^n (-1)^{i-1} \tau\para{ \tau(\beta(x_i)) \phi\para{\beta(x_1),...,\widehat{x_i},...,\beta(x_n),\alpha(x_{n+1})}} \\
+ & (-1)^n \tau\para{\tau(\alpha(x_{n+1}))\phi\para{\beta(x_1),...,\beta(x_n)}}\\
= &  (-1)^n \tau\para{\phi\para{\beta(x_1),...,\tau(\alpha(x_{n+1}))\beta(x_n)}}\\
= &  (-1)^n \tau\para{\tau(\beta(x_{n+1}))\phi\para{\beta(x_1),...,\alpha(x_n)}}\\
= & 0.
\end{align*}
\end{proof}

\begin{lemma} \label{morphinduced}
Let $A$ (resp. $B$) be a vector space, $\phi$ (resp. $\psi$) be an $n$-linear map on $A$ (resp. $B$) and $\tau$ (resp. $\sigma$) be a linear form on $A$ (resp. $B$). Let $f : A \to B$ be a linear map satisfying:
\[\forall x_1,...,x_n \in A, f\para{\phi\para{x_1,...,x_n}} = \psi \para{f(x_1),...,f(x_n)},\]
and $\tau = \sigma \circ f$, then $f$ satisfies:
\[\forall x_1,...,x_{n+1} \in A, f\para{\phi_\tau\para{x_1,...,x_{n+1}}} = \psi_\tau \para{f(x_1),...,f(x_{n+1})}.\]
\end{lemma}
\begin{proof}
For all $x_1,...,x_{n+1} \in A$, we have:
\begin{align*}
f\para{\phi_\tau\para{x_1,...,x_{n+1}}} &= \sum_{i=1}^{n+1} (-1)^{i-1} \tau(x_i)f\para{\phi\para{x_1,...,x_{i-1},x_{i+1},...,x_{n+1}}}\\
&=  \sum_{i=1}^{n+1} (-1)^{i-1} \tau(x_i)\psi\para{f(x_1),...,f(x_{i-1}),f(x_{i+1}),...,f(x_{n+1})}\\
&=  \sum_{i=1}^{n+1} (-1)^{i-1} \sigma(f(x_i))\psi\para{f(x_1),...,f(x_{i-1}),f(x_{i+1}),...,f(x_{n+1})}\\
&=  \psi_\sigma\para{f(x_1),...,f(x_{n+1})}.
\end{align*}
\end{proof}

Now, using the definitions and lemmas above, one can, under some conditions, construct an $(n+1)$-BiHom-Lie algebra using an $n$-BiHom-Lie algebra and a twisted trace, this construction is given by the following theorem:
\begin{theorem} \label{inducedbihom} 
Let $(A,\bracket{\cdot,...,\cdot},\alpha,\beta)$ be an $n$-BiHom-Lie algebra (resp. Leibniz) an $(\alpha,\beta)$-twisted $\bracket{\cdot,...,\cdot}$-trace. If the conditions:
\[\forall x,y \in A, \tau(\alpha(x))\beta(y) = \tau(\beta(x))\alpha(y), \]
\[ \tau \circ \alpha = \tau \text{ and } \tau \circ \beta = \tau, \]
are satisfied then $(A,\bracket{\cdot,...,\cdot}_\tau,\alpha,\beta)$ is an $(n+1)$-BiHom-Lie algebra (resp. Leibniz). We say that this algebra is induced by $(A,\bracket{\cdot,...,\cdot},\alpha,\beta)$.

\end{theorem}

\begin{proof}
The fact that $\alpha$ and $\beta$ commute comes from the given algebra. Lemma \ref{morphinduced} provides that $\alpha$ and $\beta$ are morphisms for the bracket $\bracket{\cdot,...,\cdot}_\tau$. We show that this bracket satisfies the BiHom-skewsymmetry and the $(n+1)$-BiHom-Jacobi (resp. BiHom Nambu) identity.
\begin{align*}
&\bracket{\beta(x_1),...,\beta(x_{j+1}),\beta(x_j),...,\beta(x_n),\alpha(x_{n+1})}_\tau\\ &= \sum_{i=1}^{j-1} (-1)^{i+1}\tau(\beta(x_i)) \bracket{\beta(x_1),...,\widehat{\beta(x_i)},...,\beta(x_{j+1}),\beta(x_j),...,\beta(x_n),\alpha(x_{n+1})} \\
&+ (-1)^{j-1} \tau(\beta(x_{j+1})) \bracket{\beta(x_1),...,\beta(x_{j-1}),\beta(x_j),\beta(x_{j+2}),...,\beta(x_n),\alpha(x_{n+1})}\\
&+ (-1)^{j} \tau(\beta(x_{j})) \bracket{\beta(x_1),...,\beta(x_{j-1}),\beta(x_{j+1}),\beta(x_{j+2}),...,\beta(x_n),\alpha(x_{n+1})}\\
&+ \sum_{i=j+2}^{n} (-1)^{i+1}\tau(\beta(x_i)) \bracket{\beta(x_1),...,\beta(x_{j+1}),\beta\para{x_j},...,\widehat{\beta(x_i)},...,\beta(x_n),\alpha(x_{n+1})} \\
&+ (-1)^n \tau(\alpha(x_{n+1})) \bracket{\beta(x_1),...,\beta(x_{j+1}),\beta(x_j),...,\beta(x_n)}\\
&= - \sum_{i=1}^{j-1} (-1)^{i+1}\tau(\beta(x_i)) \bracket{\beta(x_1),...,\widehat{\beta(x_i)},...,\beta(x_j),\beta(x_{j+1}),...,\beta(x_n),\alpha(x_{n+1})} \\
&- (-1)^{j-1} \tau(\beta(x_{j})) \bracket{\beta(x_1),...,\beta(x_{j-1}),\beta(x_{j+1}),\beta(x_{j+2}),...,\beta(x_n),\alpha(x_{n+1})}\\
&- (-1)^{j} \tau(\beta(x_{j+1})) \bracket{\beta(x_1),...,\beta(x_{j-1}),\beta(x_j),\beta(x_{j+2}),...,\beta(x_n),\alpha(x_{n+1})}\\
&- \sum_{i=j+2}^{n} (-1)^{i+1}\tau(\beta(x_i)) \bracket{\beta(x_1),...,\beta(x_j),\beta(x_{j+1}),...,\widehat{\beta(x_i)},...,\beta(x_n),\alpha(x_{n+1})} \\
&+ (-1)^n \tau(\beta(x_{n+1})) \bracket{\beta(x_1),...,\beta(x_{j+1}),\beta(x_j),...,\alpha(x_n)}\\
&= - \sum_{i=1}^n (-1)^{i-1} \tau(\beta(x_i)) \bracket{\beta(x_1),...,\widehat{\beta(x_i)},...,\beta(x_n),\alpha(x_{n+1})}\\
& - (-1)^n \tau(\beta(x_{n+1})) \bracket{\beta(x_1),...,,\beta(x_{n-1})\alpha(x_n)}\\
&= - \sum_{i=1}^n (-1)^{i-1} \tau(\beta(x_i)) \bracket{\beta(x_1),...,\widehat{\beta(x_i)},...,\beta(x_n),\alpha(x_{n+1})}\\
& - (-1)^n \tau(\alpha(x_{n+1})) \bracket{\beta(x_1),...,,\beta(x_{n-1})\beta(x_n)}\\
&= -\bracket{\beta(x_1),...,\beta(x_n),\alpha(x_{n+1})}_\tau.
\end{align*}

\begin{align*}
&\bracket{\beta(x_1),...,\beta(x_{n+1}),\alpha(x_{n})}_\tau \\
&=\sum_{i=1}^{n-1}(-1)^{i-1}\tau(\beta(x_i)) \bracket{\beta(x_1),...,\widehat{\beta(x_i)},...,\beta(x_{n+1}),\alpha(x_n)}\\
&+ (-1)^{n-1} \tau(\beta(x_{n+1})) \bracket{\beta(x_1),...,\beta(x_{n-1}),\alpha(x_n)}\\
&+ (-1)^n \tau(\alpha(x_n)) \bracket{\beta(x_1),...,\beta(x_{n-1}),\beta(x_{n+1})}\\
&= - \sum_{i=1}^{n-1}(-1)^{i-1}\tau(\beta(x_i)) \bracket{\beta(x_1),...,\widehat{\beta(x_i)},...,\beta(x_{n}),\alpha(x_{n+1})}\\
& -  (-1)^{n} \tau(\beta(x_{n+1}) \bracket{\beta(x_1),...,\beta(x_{n-1}),\alpha(x_n)}\\
&- (-1)^{n-1} \tau(\alpha(x_n)) \bracket{\beta(x_1),...,\beta(x_{n-1}),\beta(x_{n+1})}\\
&= - \sum_{i=1}^{n-1}(-1)^{i-1}\tau(\beta(x_i)) \bracket{\beta(x_1),...,\widehat{\beta(x_i)},...,\beta(x_{n}),\alpha(x_{n+1})}\\
& -  (-1)^{n} \tau(\alpha(x_{n+1}) \bracket{\beta(x_1),...,\beta(x_{n-1}),\beta(x_n)}\\
&- (-1)^{n-1} \tau(\beta(x_n)) \bracket{\beta(x_1),...,\beta(x_{n-1}),\alpha(x_{n+1})}\\
&= - \bracket{\beta(x_1),...,\beta(x_n),\alpha(x_{n+1})}_\tau.
\end{align*}

We denote by $L_1$ and $R_1$ respectively the left and the right hand sides of the $(n+1)$-BiHom-Jacobi identity, we have, for all $x_1,...,x_n,y_1,...,y_{n+1} \in A$:\\

\begin{align*}
L_1 &= \bracket{\beta^2(x_1),...,\beta^2(x_n),\bracket{\beta(y_1),...,\beta(y_n),\alpha(y_{n+1})}_\tau }_\tau\\
&= \sum_{i=1}^n (-1)^{i-1}\tau(\beta^2(x_i))\bracket{\beta^2(x_1),...,\widehat{\beta^2(x_i)},...,\beta^2(x_n),\bracket{\beta(y_1),...,\beta(y_n),\alpha(y_{n+1})}_\tau }\\
&+ (-1)^n \tau\para{\bracket{\beta(y_1),...,\beta(y_n),\alpha(y_{n+1})}_\tau} \bracket{\beta^2(x_1),...,\beta^2(x_n)}\\
&= \sum_{i=1}^n\sum_{j=1}^n (-1)^{i+j}\tau(\beta^2(x_i))\tau(\beta(y_j)) \times \\
& \qquad \times \bracket{\beta^2(x_1),...,\widehat{\beta^2(x_i)},...,\beta^2(x_n),\bracket{\beta(y_1),...,\widehat{\beta(y_j)},...,\beta(y_n),\alpha(y_{n+1})} }\\
&+ \sum_{i=1}^n (-1)^{i+n-1}\tau(\beta^2(x_i))\tau(\alpha(y_{n+1})) \bracket{\beta^2(x_1),...,\widehat{\beta^2(x_i)},...,\beta^2(x_n),\bracket{\beta(y_1),...,\beta(y_n)} }\\
&= \sum_{i=1}^n\sum_{j=1}^n (-1)^{i+j}\tau(\beta^2(x_i))\tau(\beta(y_j)) \times \\
& \qquad \times \bracket{\beta^2(x_1),...,\widehat{\beta^2(x_i)},...,\beta^2(x_n),\bracket{\beta(y_1),...,\widehat{\beta(y_j)},...,\beta(y_n),\alpha(y_{n+1})} }\\
&+ \sum_{i=1}^n (-1)^{i+n-1}\tau(\beta^2(x_i))\tau(\beta(y_{n+1})) \times \\
& \qquad \times \bracket{\beta^2(x_1),...,\widehat{\beta^2(x_i)},...,\beta^2(x_n),\bracket{\beta(y_1),...,\beta(y{n-1}),\alpha(y_n)} }\\
&= \sum_{i=1}^n\sum_{j=1}^n (-1)^{i+j}\tau(\beta^2(x_i))\tau(\beta(y_j)) \sum_{k=1}^{j-1} (-1)^{n-k} \times \\
& \qquad \times \bracket{\beta^2(y_1),...,\widehat{\beta^2(y_k)},...,\widehat{\beta^2(y_j)},...,\beta^2(y_{n+1}),\bracket{\beta(x_1),...,\widehat{\beta(x_i)},...,\beta(x_n),\alpha(y_{k})} }\\
&+ \sum_{i=1}^n\sum_{j=1}^n (-1)^{i+j}\tau(\beta^2(x_i))\tau(\beta(y_j)) \sum_{k=j+1}^{n+1} (-1)^{n-k-1}\times \\
& \qquad \times  \bracket{\beta^2(y_1),...,\widehat{\beta^2(y_k)},...,\widehat{\beta^2(y_j)},...,\beta^2(y_{n+1}),\bracket{\beta(x_1),...,\widehat{\beta(x_i)},...,\beta(x_n),\alpha(y_{k})} }\\
& + \sum_{i=1}^n (-1)^{i+n-1}\tau(\beta^2(x_i))\tau(\beta(y_{n+1})) \times \\
&\qquad \times \sum_{k=1}^n (-1)^{n-k} \bracket{\beta^2(y_1),...,\widehat{\beta^2(y_k)},...,\beta^2(y_n),\bracket{\beta(x_1),...,\widehat{\beta(x_i)},...,\beta(x_n), \alpha(y_k)} }\\
&= \sum_{i=1}^n\sum_{j=1}^{n+1} \sum_{k=1}^{j-1}(-1)^{i+j}(-1)^{n-k}\tau(\beta^2(x_i))\tau(\beta(y_j))\times \\
& \qquad \times    \bracket{\beta^2(y_1),...,\widehat{\beta^2(y_k)},...,\widehat{\beta^2(y_j)},...,\beta^2(y_{n+1}),\bracket{\beta(x_1),...,\widehat{\beta(x_i)},...,\beta(x_n),\alpha(y_{k})} }\\
&+ \sum_{i=1}^n\sum_{j=1}^{n+1} \sum_{k=j+1}^{n+1} (-1)^{i+j}(-1)^{n-k-1} \tau(\beta^2(x_i))\tau(\beta(y_j))\times \\
& \qquad \times    \bracket{\beta^2(y_1),...,\widehat{\beta^2(y_k)},...,\widehat{\beta^2(y_j)},...,\beta^2(y_{n+1}),\bracket{\beta(x_1),...,\widehat{\beta(x_i)},...,\beta(x_n),\alpha(y_{k})} }\\
\end{align*}

For the right hand side, we have:
\begin{align*}
R_1 &= \sum_{k=1}^{n+1} (-1)^{n+1-k} \bracket{\beta^2(y_1),...,\widehat{\beta^2(y_k)},...,\beta^2(y_{n+1}),\bracket{\beta(x_1),...,\beta(x_n),\alpha(y_k)}_\tau}_\tau \\
&= \sum_{k=1}^{n+1} (-1)^{n+1-k} \sum_{j=1}^{k-1} (-1)^{j-1} \tau(\beta^2(y_j))\times\\
&\qquad \times \bracket{\beta^2(y_1),...,\widehat{\beta^2(y_j)},...,\widehat{\beta^2(y_k)},...,\beta^2(y_{n+1}),\bracket{\beta(x_1),...,\beta(x_n),\alpha(y_k)}_\tau}\\
&+ \sum_{k=1}^{n+1} (-1)^{n+1-k} \sum_{j=k+1}^{n+1} (-1)^{j} \tau(\beta^2(y_j))\times \\
 &\qquad \times \bracket{\beta^2(y_1),...,\widehat{\beta^2(y_k)},...,\widehat{\beta^2(y_j)},...,\beta^2(y_{n+1}),\bracket{\beta(x_1),...,\beta(x_n),\alpha(y_k)}_\tau}\\
&+ \sum_{k=1}^{n+1} (-1)^{n+1-k}  (-1)^{n} \tau(\bracket{\beta(x_1),...,\beta(x_n),\alpha(y_k)}_\tau)\times \\
& \qquad \times  \bracket{\beta^2(y_1),...,\widehat{\beta^2(y_k)},...,,...,\beta^2(y_{n+1})}\\
&= \sum_{k=1}^{n+1} (-1)^{n+1-k} \sum_{j=1}^{k-1} \sum_{i=1}^n (-1)^{i-1} (-1)^{j-1} \tau(\beta^2(y_j)) \tau (\beta(x_i)) \times\\
&\qquad \times \bracket{\beta^2(y_1),...,\widehat{\beta^2(y_j)},...,\widehat{\beta^2(y_k)},...,\beta^2(y_{n+1}),\bracket{\beta(x_1),...,\widehat{\beta(x_i)},...,\beta(x_n),\alpha(y_k)} }\\
&+ \sum_{k=1}^{n+1} (-1)^{n+1-k} \sum_{j=k+1}^{n+1} \sum_{i=1}^n (-1)^{i-1}  (-1)^{j} \tau(\beta^2(y_j)) \tau(\beta(x_i)) \times \\
 &\qquad \times \bracket{\beta^2(y_1),...,\widehat{\beta^2(y_k)},...,\widehat{\beta^2(y_j)},...,\beta^2(y_{n+1}),\bracket{\beta(x_1),...,\widehat{\beta(x_i)},...,\beta(x_n),\alpha(y_k)} }\\
&+\sum_{k=1}^{n+1} (-1)^{n+1-k} \sum_{j=1}^{k-1} (-1)^{j-1} (-1)^n \tau(\alpha(y_k)) \tau(\beta^2(y_j))\times\\
&\qquad \times \bracket{\beta^2(y_1),...,\widehat{\beta^2(y_j)},...,\widehat{\beta^2(y_k)},...,\beta^2(y_{n+1}),\bracket{\beta(x_1),...,\beta(x_n)} }\\
&+ \sum_{k=1}^{n+1} (-1)^{n+1-k} \sum_{j=k+1}^{n+1} (-1)^{j} (-1)^n \tau(\alpha(y_k)) \tau(\beta^2(y_j))\times \\
 &\qquad \times \bracket{\beta^2(y_1),...,\widehat{\beta^2(y_k)},...,\widehat{\beta^2(y_j)},...,\beta^2(y_{n+1}),\bracket{\beta(x_1),...,\beta(x_n)} }\\
&=\sum_{i=1}^n \sum_{k=1}^{n+1}\sum_{j=1}^{k-1}  (-1)^{n+1-k}  (-1)^{i-1} (-1)^{j-1} \tau(\beta^2(y_j)) \tau (\beta(x_i)) \times\\
&\qquad \times \bracket{\beta^2(y_1),...,\widehat{\beta^2(y_j)},...,\widehat{\beta^2(y_k)},...,\beta^2(y_{n+1}),\bracket{\beta(x_1),...,\widehat{\beta(x_i)},...,\beta(x_n),\alpha(y_k)} }\\
&+ \sum_{i=1}^n \sum_{k=1}^{n+1} \sum_{j=k+1}^{n+1}  (-1)^{n+1-k} (-1)^{i-1}  (-1)^{j} \tau(\beta^2(y_j)) \tau(\beta(x_i)) \times \\
 &\qquad \times \bracket{\beta^2(y_1),...,\widehat{\beta^2(y_k)},...,\widehat{\beta^2(y_j)},...,\beta^2(y_{n+1}),\bracket{\beta(x_1),...,\widehat{\beta(x_i)},...,\beta(x_n),\alpha(y_k)} }\\
&+\sum_{k=1}^{n+1}  \sum_{j=1}^{k-1} (-1)^{n+1-k}(-1)^{j-1} (-1)^n \tau(\alpha(y_k)) \tau(\beta^2(y_j))\times\\
&\qquad \times \bracket{\beta^2(y_1),...,\widehat{\beta^2(y_j)},...,\widehat{\beta^2(y_k)},...,\beta^2(y_{n+1}),\bracket{\beta(x_1),...,\beta(x_n)} }\\
&+ \sum_{j=1}^{n+1}  \sum_{k=1}^{j-1} (-1)^{n+1-k} (-1)^{j} (-1)^n \tau(\alpha(y_k)) \tau(\beta^2(y_j))\times \\
 &\qquad \times \bracket{\beta^2(y_1),...,\widehat{\beta^2(y_k)},...,\widehat{\beta^2(y_j)},...,\beta^2(y_{n+1}),\bracket{\beta(x_1),...,\beta(x_n)} }\\
&=\sum_{i=1}^n \sum_{j=1}^{n+1}\sum_{k=j+1}^{n+1}  (-1)^{n+1-k}  (-1)^{i-1} (-1)^{j-1} \tau(\beta^2(y_j)) \tau (\beta(x_i)) \times\\
&\qquad \times \bracket{\beta^2(y_1),...,\widehat{\beta^2(y_j)},...,\widehat{\beta^2(y_k)},...,\beta^2(y_{n+1}),\bracket{\beta(x_1),...,\widehat{\beta(x_i)},...,\beta(x_n),\alpha(y_k)} }\\
&+ \sum_{i=1}^n \sum_{j=1}^{n+1} \sum_{k=1}^{j-1}  (-1)^{n+1-k} (-1)^{i-1}  (-1)^{j} \tau(\beta^2(y_j)) \tau(\beta(x_i)) \times \\
 &\qquad \times \bracket{\beta^2(y_1),...,\widehat{\beta^2(y_k)},...,\widehat{\beta^2(y_j)},...,\beta^2(y_{n+1}),\bracket{\beta(x_1),...,\widehat{\beta(x_i)},...,\beta(x_n),\alpha(y_k)} }\\
&= L_1.
\end{align*} 

The same way, for the  BiHom-Nambu identity, we have, for all $x_1,...,x_n,y_1,...,y_{n+1} \in A$:
\begin{align*}
L_2 &= \bracket{\beta^2(x_1),...,\beta^2(x_n),\bracket{\beta(y_1),...,\beta(y_n),\alpha(y_{n+1})}_\tau }_\tau\\
&= \sum_{i=1}^n (-1)^{i-1}\tau(\beta^2(x_i))\bracket{\beta^2(x_1),...,\widehat{\beta^2(x_i)},...,\beta^2(x_n),\bracket{\beta(y_1),...,\beta(y_n),\alpha(y_{n+1})}_\tau }\\
&+ (-1)^n \tau\para{\bracket{\beta(y_1),...,\beta(y_n),\alpha(y_{n+1})}_\tau} \bracket{\beta^2(x_1),...,\beta^2(x_n)}\\
&= \sum_{i=1}^n\sum_{j=1}^n (-1)^{i+j}\tau(\beta^2(x_i))\tau(\beta(y_j) \times \\
& \qquad \times\bracket{\beta^2(x_1),...,\widehat{\beta^2(x_i)},...,\beta^2(x_n),\bracket{\beta(y_1),...,\widehat{\beta(y_j)},...,\beta(y_n),\alpha(y_{n+1})} }\\
&+ \sum_{i=1}^n (-1)^{i+n-1}\tau(\beta^2(x_i))\tau(\alpha(y_{n+1})) \bracket{\beta^2(x_1),...,\widehat{\beta^2(x_i)},...,\beta^2(x_n),\bracket{\beta(y_1),...,\beta(y_n)} }\\
&= \sum_{i=1}^n\sum_{j=1}^n (-1)^{i+j}\tau(\beta^2(x_i))\tau(\beta(y_j))\times \\
& \qquad \times \bracket{\beta^2(x_1),...,\widehat{\beta^2(x_i)},...,\beta^2(x_n),\bracket{\beta(y_1),...,\widehat{\beta(y_j)},...,\beta(y_n),\alpha(y_{n+1})} }\\
&+ \sum_{i=1}^n (-1)^{i+n-1}\tau(\beta^2(x_i))\tau(\beta(y_{n+1}))\times \\
& \qquad \times  \bracket{\beta^2(x_1),...,\widehat{\beta^2(x_i)},...,\beta^2(x_n),\bracket{\beta(y_1),...,\beta(y{n-1}),\alpha(y_n)} }\\
&= \sum_{i=1}^n \sum_{j=1}^n (-1)^{i+j}\tau(\beta^2(x_i))\tau(\beta(y_j))\times \\
&\qquad \times \sum_{k=1;k\neq j}^{b+1}\bracket{\beta^2(y_1),...,\widehat{\beta^2(y_j)},...,\bracket{\beta(x_1),...,\widehat{\beta(x_i)},...,\beta(x_n),\alpha(y_k)},...,\beta^2(y_{n+1}) }\\
&+ \sum_{i=1}^n (-1)^{i+n-1}\tau(\beta^2(x_i))\tau(\beta(y_{n+1}))\times \\
&\qquad \times \sum_{k=1}^n \bracket{\beta^2(y_1),...,\bracket{\beta(x_1),...,\widehat{\beta(x_i)},...,\beta(x_{n}),\alpha(y_k)},...,\beta^2(y_n) }\\
&= \sum_{i=1}^n \sum_{j=1}^{n+1} (-1)^{i+j}\tau(\beta^2(x_i))\tau(\beta(y_j))\times \\
&\qquad \times \sum_{k=1;k\neq j}^{b+1}\bracket{\beta^2(y_1),...,\widehat{\beta^2(y_j)},...,\bracket{\beta(x_1),...,\widehat{\beta(x_i)},...,\beta(x_n),\alpha(y_k)},...,\beta^2(y_{n+1}) }\\
\end{align*}

For the right hand side, we have:
\begin{align*}
R_2 &= \sum_{k=1}^{n+1} \bracket{\beta^2(y_1),...,\bracket{\beta(x_1),...,\beta(x_n),\alpha(y_k)}_\tau,...,\beta^2(y_{n+1}) }_\tau \\
&= \sum_{k=1}^{n+1} \sum_{j=1;j\neq k}^{n+1} (-1)^{j-1} \tau\para{\beta^2(y_j)} \times \\
& \qquad \times \bracket{\beta^2(y_1),...,\widehat{\beta^2(y_j)},...,\bracket{\beta(x_1),...,\beta(x_n),\alpha(y_k)}_\tau,...,\beta^2(y_{n+1}) }\\
&= \sum_{k=1}^{n+1} \sum_{j=1;j\neq k}^{n+1} (-1)^{j-1} \tau\para{\beta^2(y_j)} \sum_{i=1}^n (-1)^{i-1}\tau\para{\beta(x_i)}\times\\
&\qquad \times \bracket{\beta^2(y_1),...,\widehat{\beta^2(y_j)},...,\bracket{\beta(x_1),...,\widehat{\beta(x_i)},...,\beta(x_n),\alpha(y_k)},...,\beta^2(y_{n+1}) }\\
&+ \sum_{k=1}^{n+1} \sum_{j=1;j\neq k}^{n+1} (-1)^{j+n-1} \tau\para{\beta^2(y_j)} \tau(\alpha(y_k))\times \\
&\qquad \times \bracket{\beta^2(y_1),...,\widehat{\beta^2(y_j)},...,\beta^2(y_{k-1}),\bracket{\beta(x_1),...,\beta(x_n)},...,\beta^2(y_{n+1}) }\\
&=\sum_{i=1}^n  \sum_{j=1}^{n+1}\sum_{k=1;k\neq j}^{n+1} (-1)^{i+j} \tau\para{\beta(x_i)}\tau\para{\beta^2(y_j)}\times\\
&\qquad \times \bracket{\beta^2(y_1),...,\widehat{\beta^2(y_j)},...,\bracket{\beta(x_1),...,\widehat{\beta(x_i)},...,\beta(x_n),\alpha(y_k)},...,\beta^2(y_{n+1}) }\\
&+ \sum_{k=1}^{n+1} \sum_{j=1;j\neq k}^{n+1} (-1)^{j+n-1} \tau\para{\beta^2(y_j)} \tau(\beta(y_k))\times\\
&\qquad \times \bracket{\beta^2(y_1),...,\widehat{\beta^2(y_j)},...,\beta^2(y_{k-1}),\bracket{\beta(x_1),...,\beta(x_n)},...,\alpha\circ\beta(y_{n+1}) }\\
&= L_2 \\
&+ \sum_{k=1}^{n+1} \sum_{j=1}^{k-1} (-1)^{j+n-1}\tau\para{\beta^2(y_j)} \tau(\beta(y_k))\times \\
&\qquad \times \bracket{\beta^2(y_1),...,\widehat{\beta^2(y_j)},...,\beta^2(y_{k-1}),\bracket{\beta(x_1),...,\beta(x_n)},...,\alpha\circ\beta(y_{n+1}) }\\
&+ \sum_{k=1}^{n+1} \sum_{j=k+1}^{n+1} (-1)^{j+n-1} \tau\para{\beta^2(y_j)} \tau(\beta(y_k)) \times \\
&\qquad \times \bracket{\beta^2(y_1),...,\beta^2(y_{k-1}),\bracket{\beta(x_1),...,\beta(x_n)},...,\widehat{\beta^2(y_j)},...,\alpha\circ\beta(y_{n+1}) }\\
&= L_2 \\
&+ \sum_{k=1}^{n+1} \sum_{j=1}^{k-1} (-1)^{j+n-1} \tau\para{\beta^2(y_j)} \tau(\beta(y_k)) \times \\
&\qquad \times \bracket{\beta^2(y_1),...,\widehat{\beta^2(y_j)},...,\beta^2(y_{k-1}),\bracket{\beta(x_1),...,\beta(x_n)},...,\alpha\circ\beta(y_{n+1}) }\\
&+ \sum_{k=1}^{n+1} \sum_{j=k+1}^{n+1} (-1)^{j+n-1} \tau\para{\beta^2(y_j)} \tau(\beta(y_k)) \times \\
&\qquad \times \bracket{\beta^2(y_1),...,\beta^2(y_{k-1}),\bracket{\beta(x_1),...,\beta(x_n)},...,\widehat{\beta^2(y_j)},...,\alpha\circ\beta(y_{n+1}) }\\
\end{align*}

\begin{remark}\label{weakver}
If one considers a weaker version of $n$-BiHom-Lie algebras, namely considering the maps $\alpha$ and $\beta$ to be any linear maps instead of being algebra morphisms, then one can replace the conditions
\[ \tau\circ \alpha = \tau \text{ and } \tau\circ \beta = \tau \]
 by the condition
 \[ \tau(\beta(x))\tau(\beta^2(y)) = \tau(\beta^2(x))\tau(\beta(y)). \]

\end{remark}
\end{proof}

We study now the conditions on $\alpha$ and $\beta$ in the Theorem \ref{inducedbihom}, to see how restrictive they are for the choice of an $n$-BiHom-Lie algebra to use for the construction.

We consider an $n$-BiHom-Lie algebra $\para{A,\bracket{\cdot,...,\cdot},\alpha,\beta}$ and $\tau : A \to \K$ a twisted trace such that none of $\bracket{\cdot,...,\cdot}$, $\tau$, $\alpha$ and $\beta$ is identically zero.  Suppose that they satisfy the conditions of Theorem \ref{inducedbihom}, then we get the following:
\[ \alpha(\ker(\tau)) \subseteq \ker(\tau) \text{ and } \beta(\ker(\tau)) \subseteq \ker(\tau). \]
The two conditions of Theorem \ref{inducedbihom} put together also lead to the following:

Let $x\in A$ such that $\tau(x)\neq 0$, for all $y\in A$, we have:
\begin{align*}
\tau(\alpha(x))\beta(y) = \tau(\beta(x))\alpha(y) & \implies \tau(x)\beta(y) = \tau(x)\alpha(y) \\
& \implies \beta(y) = \alpha(y). \qquad \text{ (since $\tau(x)\neq 0$)}
\end{align*}
If one drops some conditions as explained in Remark \ref{weakver}, it is possible to get a less restrictive situation.

\begin{example}\label{ex}
Let $A$ be a vector space, $\dim A = n=3$ with basis $(e_i)_{1\leq i \leq n}$ and the linear maps $\bracket{\cdot,\cdot} : A \otimes A \to A$ and $\alpha,\beta : A \to A$ given by:
\[\bracket{e_i,e_j}=\sum_{k=1}^n c_{ij}^k e_k, \forall i,j : 1\leq i,j \leq n,\]
and
\[ [\alpha]=(a_{i,j})_{1\leq i,j \leq n} ; [\beta]=(b_{i,j})_{1\leq i,j \leq n}. \]
We suppose $[\alpha]$ and $[\beta]$ to be diagonal, which is a sufficient condition for $\alpha \circ \beta = \beta \circ \alpha$. The remaining conditions for $(A,\bracket{\cdot,\cdot},\alpha,\beta)$ to be a BiHom-Lie-Leibniz  algebra are given, under the assumptions above, by:
\begin{itemize}
\item $\alpha,\beta$-skewsymmetry:
\[ \sum_{k=1}^n \sum_{l=1}^3 (a_{ki}b_{lj}+a_{kj}b_{li})c_{lk}^p = 0, \forall i,j,p : 1\leq i,j,p \leq n. \]
\item Multiplicativity
\[  \sum_{k=1}^n c_{ij}^k a_{k,q} - \sum_{k=1}^n \sum_{l=1}^n a_{k,i}a_{lj}c_{kl}^q = 0, \forall i,j,q : 1 \leq i,j,q \leq n.  \]
\item BiHom-Leibniz identity:
\begin{align*}
\sum_{l=1}^n \sum_{p=1}^n \sum_{q=1}^n \sum_{r=1}^n \sum_{s=1}^n (b_{pl}b_{li}b_{qj}a_{rk} c_{qr}^s c_{ps}^t  - b_{pl}b_{lk}b_{qi}a_{rj} c_{qr}^s c_{sp}^t - b_{pl}b_{lj}b_{qi}a_{rk} c_{qr}^s c_{ps}^t)=0,& \\
 \forall i,j,k,t : 1 \leq i,j,k,t \leq n. &
 \end{align*}
\end{itemize}
We solve these equations, starting with $\alpha,\beta$-skewsymmetry, then BiHom-Leibniz identity and multiplicativity, while choosing at each step solutions where none of $\bracket{\cdot,\cdot}$, $\alpha$ and $\beta$ are zero. This way, we can find examples of BiHom-Lie-Leibniz  algebras.

We consider now a linear map $\tau : A \to \K$, given by:
\[ \tau(e_i)=t_i, \forall i : 1\leq i \leq n. \]
The conditions of Theorem \ref{inducedbihom} become:
\[ \sum_{k=1}^n \sum_{l=1}^n \sum_{o=1}^n b_{ki} a_{lj} c_{cl}^o t_o = 0, \forall i,j : 1\leq i,j \leq 3. \]
\[ t_i-\sum_{k=1}^n a_{ki}t_k = 0, \forall i : 1\leq i \leq n. \]
\[ t_i-\sum_{k=1}^n a_{ki}t_k = 0, \forall i : 1\leq i \leq n. \]
\[ \sum_{k=1}^n (a_{ki}b_{lj}t_k - b_{ki}a_{lj}t_k)=0, \forall i,j,l : 1\leq i,j,l \leq n. \]
Solving these equations gives conditions on $\tau$, $\alpha$, $\beta$ (and sometimes on $\bracket{\cdot,\cdot}$) such that one can construct the induced Leibniz 3-BiHom-Lie algebra. We give now some examples obtained using this procedure:
\begin{enumerate}
\item The bracket and the structure maps are given by: 
\[ [e_1, e_3]=c_1 e_3 ; [e_2,e_3]=c_2 e_3 ; [e_3,e_1]=c_3 e_3 ; [e_3,e_2]=c_4 e_3 ; [e_3,e_3]=c_5 e_3. \]
\[ [\alpha]=\begin{pmatrix} a_1 & 0 & 0\\ 0 & a_2 & 0 \\ 0 & 0 & 0 \end{pmatrix} ; [\beta]=\begin{pmatrix} b_1 & 0 & 0\\ 0 & b_2 & 0 \\ 0 & 0 & 0 \end{pmatrix}. \]
One solution to have the conditions of Theorem \ref{inducedbihom} is given by:
\[ t_2=t_3=0 ; a_1=b_1=1 ; a_2=b_2. \]
We get a Leibniz $3$-BiHom-Lie algebra defined by:
\begin{align*}
[e_1,e_2,e_3]&=t_1 c_2 e_3 ;\ [e_1,e_3,e_1]=(t_1 c_1+t_1c_3) e_3 ;\ [e_1,e_3,e_2]=t_1 c_4 e_3 ;\ [e_1,e_3,e_3]=t_1 c_5 e_3 ;\\
[e_2,e_1,e_3]&=-t_1 c_2 e_3 ;\ [e_2,e_3,e_1]=t_1 c_2 e_3 ;\ [e_3,e_1,e_2]=-t_1 c_4 e_3 ;\ [e_3,e_1,e_3]=-t_1 c_5 e_3 ;\\
[e_3,e_2,e_1]&=t_1 c_4 e_3 ;\ [e_3,e_3,e_3]=t_1 c_5 e_3.
\end{align*}
\[ [\alpha]=\begin{pmatrix} 1 & 0 & 0\\ 0 & a_2 & 0 \\ 0 & 0 & 0 \end{pmatrix} ; [\beta]=\begin{pmatrix} 1 & 0 & 0\\ 0 & a_2 & 0 \\ 0 & 0 & 0 \end{pmatrix}. \]
\item The bracket and the structure maps are given by: 
\[ [e_1, e_2]=c_1 e_2+c_2 e_3 ; [e_1,e_3]=c_3 e_3 ; [e_2,e_2]=c_4 e_2+c_5 e_3 ; [e_3,e_2]=c_6 e_3. \]
\[ [\alpha]=\begin{pmatrix} a_1 & 0 & 0\\ 0 & 0 & 0 \\ 0 & 0 & 0 \end{pmatrix} ; [\beta]=\begin{pmatrix} 1 & 0 & 0\\ 0 & 1 & 0 \\ 0 & 0 & 1 \end{pmatrix}. \]
In this case, there is no (nonzero) $\tau$ satisfying the conditions of Theorem \ref{inducedbihom}.
\item The bracket and the structure maps are given by: 
\[ [e_1, e_3]=c_1 e_3 ; [e_3,e_1]=-c_1 e_3. \]
\[ [\alpha]=\begin{pmatrix} 1 & 0 & 0\\ 0 & a_1 & 0 \\ 0 & 0 & a_2 \end{pmatrix} ; [\beta]=\begin{pmatrix} 1 & 0 & 0\\ 0 & b_1 & 0 \\ 0 & 0 & a_2 \end{pmatrix}. \]
In this case, when solving equations for $\tau$, $\alpha$ and $\beta$ to satisfy the conditions of Theorem \ref{inducedbihom}, we get either $\tau=0$ or $\alpha=\beta$. In the second case, the chosen algebra becomes a Hom-Lie algebra (see \cite{Bihom1}) if $a_1,a_2$ are non-zero. Namely, all possible solutions (where $\tau\neq 0$) are:
\[a_1 = 1 ; b_1 = 1 ; t_3 = 0.\]
\[ b_1=a_1 ; t_2=0 ; t_3=0. \]
The second solution leads to the induced algebra's bracket being zero. Let us look at the first one. The structure maps, under these conditions become:
\[ [\alpha]=\begin{pmatrix} 1 & 0 & 0\\ 0 & a_1 & 0 \\ 0 & 0 & a_2 \end{pmatrix} ; [\beta]=\begin{pmatrix} 1 & 0 & 0\\ 0 & b_1 & 0 \\ 0 & 0 & a_2 \end{pmatrix}, \]
and the bracket of the induced algebra is skewsymmetric and is given by:
\[ [e_1,e_2,e_3]=t_2 c_1 e_3. \]
\end{enumerate}
\end{example}

\begin{example}
Now let us look at a case where we drop the condition that $\alpha$ and $\beta$ need to be morphisms of the induced algebra, as in Remark \ref{weakver}. The equations $\tau \circ \alpha = \tau$ and $\tau \circ \beta = \tau$ will be replaced by the following: \[\forall x,y \in A, \tau (\beta^2(x))\tau(\beta(y))=\tau (\beta(x))\tau(\beta^2(y)). \]
In terms of structure constants, it takes the following form:
\[\sum_{k=1}^{n}\sum_{o=1}^{n}\sum_{l=1}^{n}\left( \left( b_{kl}b_{li}b_{oj}-b_{ki}b_{ol}b_{lj}\right) t_{k}t_{o}\right) =0.\]
We consider the BiHom-Lie-Leibniz algebra, generated in the same way as above, given by:
\[ [e_2, e_1]=c_1 e_2 ; [e_2,e_2]=c_2 e_2 ; [e_2,e_3]=c_3 e_2. \]
\[ [\alpha]=\begin{pmatrix} a_1 & 0 & 0\\ 0 & 0 & 0 \\ 0 & 0 & a_2 \end{pmatrix} ; [\beta]=\begin{pmatrix} b_1 & 0 & 0\\ 0 & 0 & 0 \\ 0 & 0 & b_2 \end{pmatrix}. \]
One solution to have the conditions of Theorem \ref{inducedbihom} and Remark \ref{weakver} is given by:
\[ t_1=t_3=0. \]
And we get the following induced $3$-BiHom-Lie-Leibniz algebra (without the multiplicativity property for $\alpha$ and $\beta$) defined by the bracket is given by:
\[ [e_2,e_1,e_2]=t_2 c_1 e_2;  [e_2,e_2,e_2]=t_2 c_2 e_2;  [e_2,e_3,e_2]=t_2 c_3 e_2, \]
together with the same linear maps $\alpha$ and $\beta$.
\end{example}


\begin{small}

\section*{Acknowledgement}
A. Kitouni is grateful to Division of Applied Mathematics, the research environment Mathematics and Applied Mathematics (MAM) at the School of Education, Culture and Communication at
 M{\"a}lardalen University,
V{\"a}ster{\aa}s,  Sweden for providing support and excellent research environment during his visits to M{\"a}lardalen University when part of the work on this paper has been performed.

\end{small}

\end{document}